\documentclass[11pt]{amsart}

\usepackage{amssymb,amsmath,amsthm,cancel}

\newcommand{\inters}{\cap}
\newcommand{\from}{\colon}
\renewcommand{\implies}{\Rightarrow}
\newcommand{\E}{\mathrel{E}}
\newcommand{\G}{\mathrel{G}}
\newcommand{\mrel}[1]{\mathrel{#1}}
\renewcommand{\subset}{\subseteq}
\renewcommand{\supset}{\supseteq}
\newcommand{\comp}[1]{{#1}^{\text{c}}} 
\newcommand{\define}[1]{\textbf{#1}}

\newcommand{\union}{\cup}
\newcommand{\bigunion}{\bigcup}
\newcommand{\disjointunion}{\sqcup}

\newcommand{\biginters}{\bigcap}

\DeclareMathOperator{\dom}{dom}
\DeclareMathOperator{\ran}{ran}
\DeclareMathOperator{\Free}{Free}

\newcommand{\Z}{\mathbb{Z}}
\newcommand{\Zt}{(\Z/2\Z)}
\newcommand{\R}{\mathbb{R}}
\newcommand{\N}{\mathbb{N}}
\newcommand{\F}{\mathbb{F}}

\theoremstyle{plain}
\newtheorem{thm}{Theorem}[section]
\newtheorem*{thm*}{Theorem}
\newtheorem{lemma}[thm]{Lemma}
\newtheorem{prop}[thm]{Proposition}

\newtheorem{question}[thm]{Question}

\theoremstyle{definition}

\begin{document}

\author{Andrew S. Marks}
\thanks{The author is partially supported by the National Science Foundation under DMS-1204907, and the John Templeton foundation under Award No.\
15619.}
\thanks{The author
would also like to thank the Institute for Mathematical Sciences and the
Department of Mathematics of the National University of Singapore and the
John Templeton Foundation for their support to attend the 2012 summer
school in logic, where the main lemma of this paper was conceived.}

\address{Department of Mathematics, California Institute of Technology}
\email{marks@caltech.edu}
\title{A determinacy approach to Borel combinatorics}

\begin{abstract}

We introduce a new method, involving infinite games and Borel determinacy,
which we use to answer several well-known questions in Borel combinatorics.
\end{abstract}

\maketitle

\section{Introduction}
\label{sec:intro}

A \define{Borel graph} on a standard Borel space $X$ is a symmetric
irreflexive relation $G$ on $X$ that is Borel as a subset of $X \times X$.
We call elements of $X$ vertices, and if $x,y \in X$ and $x \G y$ then we
say that $x$ and $y$ are \define{neighbors}, or are \define{adjacent}. The
\define{degree} of a vertex is its number of neighbors, and a
graph is said to have \define{degree $\leq n$} if each of its vertices has
degree $\leq n$. A
graph is said to be \define{regular} if all of its vertices have the same
number of neighbors, and is \define{$n$-regular} if this number is $n$.

Graph coloring is a typical problem studied in the field of Borel
combinatorics, where a \define{Borel coloring} of a Borel graph $G$ on $X$ is a
Borel function $c: X \to Y$ from the vertices of $G$ to a standard Borel space $Y$
such that if $x \G y$, then $c(x) \neq c(y)$. The \define{Borel chromatic
number} $\chi_B(G)$ of $G$ is the least cardinality of a standard Borel
space $Y$ such that $G$ has a Borel coloring with codomain $Y$. The first
systematic study of Borel chromatic numbers was done by Kechris,
Solecki, and Todorcevic~\cite{MR1667145}. 
Since then, fruitful connections have been found between the study of
Borel chromatic numbers and other areas of mathematics such as ergodic
theory and dynamics~\cite{MR3019078, ConleyKechrisTuckerDrob}, and dichotomies in descriptive
set theory~\cite{MillerBull}.

If $G$ is a Borel graph, then it is clear that $\chi(G) \leq \chi_B(G)$,
where $\chi(G)$ is the usual chromatic number of $G$. However, $\chi(G)$
and $\chi_B(G)$ may differ quite wildly. For instance, Kechris,
Solecki, and Todorcevic~\cite{MR1667145} show the existence of an
acyclic Borel graph $G_0$
(so $\chi(G_0) = 2$) for which $\chi_B(G_0) = 2^{\aleph_0}$. Nevertheless, in some
respects the Borel chromatic number of a graph is quite similar
to the usual chromatic number. For example, we have the following analogue
of an obvious classical fact:

\begin{thm}[Kechris, Solecki, and Todorcevic~\cite{MR1667145}]\label{KST_ncoloring}
  If $G$ is a Borel graph of degree $\leq n$,
then $\chi_B(G) \leq n + 1$. 
\end{thm}

We will be interested in Borel graphs that arise from free Borel actions of
countable marked groups. Recall that a \define{marked group} is a group
with a specified set of generators. We assume throughout this paper that
the set of generators of a group does not include the identity. 
Let $\Gamma$ be a countable discrete group, and $X$ be a
standard Borel space. We endow the space $X^\Gamma$ of functions from
$\Gamma$ to $X$ with the usual product Borel structure (arising from the
product topology) so that $X^\Gamma$ is also
a standard Borel space. The \define{left shift action} of $\Gamma$ on
$X^\Gamma$ is defined by $\alpha \cdot y(\beta) = y(\alpha^{-1} \beta)$ for
$y \in X^\Gamma$ and $\alpha, \beta \in \Gamma$. The \define{free part} of this
action, denoted $\Free(X^{\Gamma})$, is the set of $y \in X^{\Gamma}$ such that $\gamma \cdot y \neq
y$ for all nonidentity $\gamma \in \Gamma$. Now we define $G(\Gamma,X)$ to
be the Borel graph on $\Free(X^{\Gamma})$ where for $x, y \in
\Free(X^{\Gamma})$, we have 
$x
\mrel{G(\Gamma,X)} y$ if there is a generator $\gamma \in \Gamma$ such that
$\gamma \cdot x = y$ or $\gamma \cdot y = x$. Hence, each connected
component of $G(\Gamma,X)$ is an isomorphic copy of the Cayley graph of
$\Gamma$. We will only be interested in $G(\Gamma,X)$ when $\Gamma$ is
finitely generated; an easy Baire category argument shows that if $\Gamma$
has infinitely many generators, then $\chi_B(G(\Gamma,2)) = 2^{\aleph_0}$
(see \cite{MR1667145}).

If $\Gamma$ is a marked countable group, 
then $G(\Gamma,\N)$ attains the maximum Borel chromatic number of all
graphs generated by a free Borel action of $\Gamma$. That is, suppose
we have any free Borel action of $\Gamma$ on a standard Borel space $X$, to
which we associate the Borel graph $G^X_\Gamma$ on $X$ where $x
\G^X_\Gamma
y$ if there is a generator $\gamma$ of $\Gamma$ such that $\gamma \cdot x =
y$ or $\gamma \cdot y = x$. Then $\chi_B(G^X_\Gamma) \leq
\chi_B(G(\Gamma,\N))$. This is trivial when $\Gamma$ is finite.
When $\Gamma$ is infinite, it follows from 
{\cite[Theorem 5.4]{MR1900547}}; since the action of $\Gamma$ on $X$ is
free, the function constructed there will an injective
equivariant function from $X$ into
$\Free(\N^\Gamma)$. Recall that if $\Gamma$ acts on the spaces
$X$ and $Y$, then a function $f: X \to Y$ is said to be
\define{equivariant} if for all $\gamma \in \Gamma$ we have that $\gamma
\cdot f(x) = f(\gamma \cdot x)$.

Our first result is a theorem describing how the
Borel chromatic number of $G(\Gamma,\N)$ behaves with respect to free products
(see \ref{chromnum_free_product}).
We stipulate that if $\Gamma$ and $\Delta$ are marked groups, then their
free product $\Gamma * \Delta$ is the marked group generated by the union
of the generators of $\Gamma$ and $\Delta$. 

\begin{thm}\label{free_product_chromnum_intro}
  If $\Gamma$ and $\Delta$ are finitely generated marked groups,
  then 
  \[\chi_B(G(\Gamma*\Delta,\N)) \geq \chi_B(G(\Gamma,\N)) +
  \chi_B(G(\Delta,\N))
  -1\]
\end{thm}

It has been an open question what Borel chromatic numbers can be attained
by an $n$-regular acyclic Borel graph, and whether the upper bound given by
Theorem~\ref{KST_ncoloring} is optimal for such graphs.
Several prior results
exist along these lines.
For
$2$-regular acyclic graphs, we have that 
$\chi_B(G(\Z,2)) = 3$ by~\cite{MR1667145}. More recently, Conley and
Kechris~\cite{MR3019078} have shown that for the free group on $n$
generators, $\chi_B(G(\F_n,2)) \geq \frac{n + 2\sqrt{n-1}}{2 \sqrt{n-1}}$,
and Lyons and Nazarov~\cite{MR2825538} have pointed out that results of Frieze and
Luczak~\cite{MR1142268}
imply that $\chi_B(G(\F_n,2)) \geq \frac{n}{
\log 2n}$ for sufficiently large $n$. 

Using Theorem~\ref{free_product_chromnum_intro}, we answer this question and show
that for every $n$ there exists an $n$-regular acyclic Borel graph with
Borel chromatic number equal to $n+1$. Indeed, if $(\Z/2\Z)^{*n}$ is the
free product of $n$ copies of $\Z/2\Z$, then $\chi_B\left( G \left(
(\Z/2\Z)^{*n},\N\right)\right) = n+1$, since 
Theorem~\ref{free_product_chromnum_intro} gives a tight lower bound to the upper
bound of Theorem~\ref{KST_ncoloring}. Similarly, for 
the free group on $n$ generators, we have $\chi_B(G(\F_n,\N)) = 2n + 1$.

Further, we can give a complete description of the Borel
chromatic numbers that can be attained by an $n$-regular acyclic Borel graph; they are exactly
those allowed by Theorem~\ref{KST_ncoloring} (see \ref{possible_chromnum}):

\begin{thm}\label{intro_possible_chromnum}
  For every $n \geq 1$ and every $m \in \{2, \ldots, n+1\}$, there is a
  $n$-regular acyclic Borel graph $G$ with $\chi_B(G) = m$. 
\end{thm}

In the theorem above, $G$ may be chosen to arise from a free
Borel action of $(\Z/2\Z)^{*n}$. 

Our results above involve graphs of the form $G(\Gamma,\N)$. Answering a
question originally posed in an early version of this paper, Seward and Tucker-Drob~\cite{1402.4184} have shown that for all marked groups $\Gamma$, and all $n \geq
2$, we have $\chi_B(G(\Gamma,\N)) = \chi_B(G(\Gamma,n))$. Hence, our
results apply to graphs of the form $G(\Gamma,2)$ as well.

Next, we turn to Borel edge colorings. Let $G$ be a Borel graph on a
standard Borel space $X$. If $x, y \in X$ then we say the set $\{x,y\}$ is an edge of $G$
if $x \G y$. The \define{line graph} $\check{G}$ of $G$
is the graph whose vertices are the edges of $G$, and where distinct
$\{x,y\}$ and $\{z,w\}$ are adjacent if $\{x,y\}
\inters \{z,w\} \neq \emptyset$. A \define{Borel edge coloring} of $G$ is
defined to be a Borel coloring of $\check{G}$. The \define{Borel edge
chromatic number} of a Borel graph $G$, denoted $\chi_B'(G)$, is the Borel
chromatic number of its line graph. 

It is a classical theorem of Vizing (see e.g.\ {\cite[Theorem
5.3.2]{MR2159259}}) that every $n$-regular
graph
has an edge coloring with $n+1$ colors. Kechris, Solecki and
Todorcevic have asked if the analogous fact is true for $n$-regular
Borel
graphs~{\cite[page 15]{MR1667145}}. More recently, this question has
attracted some interest from the study of graph limits
\cite{MR2644905} {\cite[Remark 3.8]{1205.4356}}.
We show that this question has a negative answer, and we calculate exactly what Borel edge
chromatic numbers can be attained by an $n$-regular Borel graph. Note that
if $G$ is an $n$-regular Borel graph, then since $\check{G}$ is $2n-2$
regular, we see that $\chi_B(\check{G}) \leq 2n-1$ by
Theorem~\ref{KST_ncoloring}. We show that this obvious upper bound can be
achieved, even using acyclic and Borel bipartite graphs (see
\ref{edge_coloring_a}). Recall that a
\define{Borel bipartite graph} is a Borel graph $G$ on $X$ for which there is a
partition of $X$ into two Borel sets $A$ and $B$ such that if $x \G y$, then
either $x \in A$ and $y \in B$, or $x \in B$ and $y \in A$. 

\begin{thm}\label{edge_coloring}
  For every $n \geq 1$ and every $m \in \{n, \ldots, 2n-1\}$, there is an 
  $n$-regular acyclic Borel bipartite graph $G$ such that
  $\chi'_B(G) = m$.
\end{thm}

A \define{Borel perfect matching} of a Borel graph $G$ is a Borel subset
$M$ of the edges of $G$ such that every vertex of $G$ is incident to
exactly one edge of $M$. In his 1993 problem list, Miller asked whether
there is a Borel analogue of Hall's theorem for
matchings~{\cite[15.10]{MR1234292}}. 
Laczkovich~\cite{MR947676} showed the existence of a 
$2$-regular Borel bipartite graph with no Borel perfect matching, and this result was
extended to give examples of $n$-regular Borel bipartite graphs with no
Borel perfect matchings by Conley and Kechris~\cite{MR3019078} when $n$
is even. However, the case for odd $n > 1$ had remained open. We obtain the
following (see \ref{matchings_a}):

\begin{thm}\label{matchings}
  For every $n > 1$, there exists an $n$-regular acyclic Borel bipartite
  graph with no Borel perfect matching.
\end{thm}

Some positive results on measurable matchings have recently been obtained by Lyons and
Nazarov~\cite{MR2825538}. Among their results, they show that the graph we use to prove
the case $n=3$ in 
Theorem~\ref{matchings} has a Borel matching modulo a null set with respect
to a natural measure. 
Further work on matchings in the measurable context has been done by 
Csoka and Lippner~\cite{1211.2374}. The measurable analogue of
Theorem~\ref{matchings} for odd $n$ remains open.

Both Theorems~\ref{edge_coloring} and \ref{matchings} are corollaries of
the following result on Borel disjoint complete sections (see \ref{nodcs}). Suppose $X$ is a standard
Borel space, and $E$ is an equivalence relation on $X$. Then a
\define{complete section} for $E$ is a set $A \subset X$ that meets every
equivalence class of $E$. Now suppose that $F$ is also an equivalence
relation on $X$. Then say that $E$ and $F$ have \define{Borel disjoint
complete sections} if there exist disjoint Borel sets $A, B \subset X$ such
that $A$ is a complete section for $E$ and $B$ is a complete section for
$F$. 

\begin{thm}\label{nodcs_intro}
  Let $\Gamma$ and $\Delta$ be countable groups. Let $E_\Gamma$ be the
  equivalence relation on $\Free(\N^{\Gamma * \Delta})$ where $x \E_\Gamma
  y$ if there exists a $\gamma \in \Gamma$ such that $\gamma \cdot x = y$.
  Define $E_\Delta$ analogously. Then $E_\Gamma$ and $E_\Delta$ do not have
  Borel disjoint complete sections.
\end{thm}

Theorems~\ref{free_product_chromnum_intro}-\ref{nodcs_intro} above all
follow from a single lemma which we prove in Section~\ref{sec:main}.
Unusually for the subject, this lemma is proved using a direct application
of Borel determinacy. 
Borel determinacy is the theorem, due to
Martin~\cite{MR0403976}, that there is a winning strategy for one of the
players in every infinite two-player game of perfect information with a
Borel payoff set. We will use the determinacy of a class of games for
constructing functions from free products of countable groups to $\N$.
Thus, we are also interested in differences between the results proved
using our new technique, and what can be shown using more standard tools
such as measure theory and Baire category, which have been a mainstay of
proofs in Borel combinatorics. 

Here, Theorem~\ref{nodcs_intro} provides a nice contrast because it is not
true in the context of measure or category, except for the single case
where $\Gamma = \Delta = \Z/2\Z$. Indeed, we have the following more
general theorem (see \ref{dcs_equiv_1}). Recall that a
\define{countable Borel equivalence relation} on a standard Borel space $X$ is an equivalence relation on $X$ that is Borel as a
subset of $X \times X$ and whose equivalence classes are countable.
$E_\Gamma$ and $E_\Delta$ in Theorem~\ref{nodcs_intro} are examples of
countable Borel equivalence relations. 

\begin{thm}\label{intro_dcs_measure_cat}
Suppose $E$ and $F$ are countable Borel equivalence relations on a standard
Borel space $X$ such that every equivalence class of $E$ has cardinality
$\geq 3$ and every equivalence class of $F$ has cardinality $\geq 2$. Then
$E$ and $F$ have Borel disjoint complete sections modulo a null set or
meager set with respect to any Borel probability measure on $X$ or Polish
topology realizing the standard Borel structure of $X$. 
\end{thm}

As we will see, the idea of disjoint complete sections turns out to be
surprisingly robust, as evidenced by a large number of equivalent
formulations which we give in Theorems \ref{dcs_equiv_1} and
\ref{dcs_equiv_2}
later in the paper.
Using the existence of disjoint complete sections in the context of measure
and category, we also show the following, which contrasts nicely with
Theorem~\ref{edge_coloring}, and demonstrates that it can not be proved
using measure-theoretic or Baire category techniques (see
\ref{measurable_3regular_edge}):

\begin{thm}\label{thm:measurable_3_regular_edge}
  Suppose $G$ is a $3$-regular Borel bipartite graph on $X$. Then $G$ has a
  Borel edge coloring with $4$ colors modulo a null set or meager set with
  respect to any Borel probability measure on $X$ or Polish topology
  realizing the standard Borel structure of $X$. 
\end{thm}

Finally, in recent joint work with Clinton Conley and Robin
Tucker-Drob~\cite{CMT-D},
we have shown that for every $n \geq 3$ and every Borel
graph $G$ of degree $\leq n$ on a standard Borel space $X$, if $G$ does not
contain a complete graph on
$n+1$ vertices, then there is a $\mu$-measurable
$n$-coloring of $G$ with respect to any Borel probability measure $\mu$ on
$X$ and a Baire measurable $n$-coloring of $G$ with respect to every
compatible Polish topology on $X$.
Hence,
Theorem~\ref{intro_possible_chromnum} can not be proved using pure measure
theoretic or Baire category arguments, except in the exceptional case $n=2$.

\subsection{Notation and conventions}
\label{subsec:notation}

Our basic reference for descriptive set theory is \cite{MR1321597}. 
Throughout we will use $X$, $Y$, and $Z$ to denote standard Borel spaces,
$x$, $y$, and $z$ for elements of such spaces, and $A$, $B$, and $C$ for
subsets of standard Borel spaces (which will generally be Borel). Given a
subset $A$ of a standard Borel space, we let $\comp{A}$ denote its
complement. 

We will use $E$ and $F$ for countable Borel equivalence
relations, and $G$ and $H$ for Borel graphs. We will use $f$, $g$, and $h$
to denote functions between standard Borel spaces, and $c$ for Borel colorings. 
$\Gamma$ and $\Delta$ will be used to denote countable groups, and $\alpha$,
$\beta$, $\gamma$, and $\delta$ will be their elements. We will use $e$ for
the identity of a group. By countable group, we will always mean
countable discrete group.

If $E$ is a countable Borel equivalence relation on a standard Borel
space $X$, then $A \subset X$ is said to be \define{$E$-invariant} if $x \in A$ and
$x \E y$ implies $y \in A$. If $B$ is a subset of $X$,
then we will often consider the largest $E$-invariant subset of $B$.
Precisely, this is the set $A$ of $x \in X$ such that for 
all $y \in X$ where $y \E x$, we have $y \in B$. 

\subsection{Acknowledgments}

Section 4 of this paper is taken from the author's thesis, which was
written under the excellent direction of Ted Slaman. The author would like
to thank Professor Slaman for many years of wise advice. The author would
also like to thank Clinton Conley, Alekos Kechris, Benjamin Miller, Anush
Tserunyan, Robin Tucker-Drob, and Jay Williams for providing helpful
feedback and suggestions throughout the development of this paper. Finally,
the author would like to thank the referee for many helpful suggestions.

\section{The main lemma}
\label{sec:main}

Let $\Gamma$ and $\Delta$ be disjoint countable groups, and let $\Gamma *
\Delta$ be their free product. Each nonidentity element of $\Gamma *
\Delta$ can be uniquely written as a finite product of either the form
$\gamma_{i_0} \delta_{i_1} \gamma_{i_2} \delta_{i_3} \ldots$ or
$\delta_{i_0} \gamma_{i_1} \delta_{i_2} \gamma_{i_3} \ldots$, where
$\gamma_i \in \Gamma$ and $\delta_i \in \Delta$ are nonidentity elements
for all $i$. Words of the former form we call \define{$\Gamma$-words}, and
words of the latter form we call \define{$\Delta$-words}. Our proof will
use games for building an element $y \in \N^{\Gamma * \Delta}$ where player
I defines $y$ on $\Gamma$-words and player II defines $y$ on
$\Delta$-words.

The following simple observation will let us combine winning strategies in
these games in a useful way. Let $W_\Gamma$ and $W_\Delta$ be the sets of
$\Gamma$-words and $\Delta$-words respectively. Then for distinct
$\gamma, \gamma' \in \Gamma$ we have that $\gamma W_\Delta$ and
$\gamma' W_\Delta$ are disjoint, and the analogous fact is true when the
roles of $\Gamma$ and $\Delta$ are switched. 

We now proceed to our main lemma. Note that both
  $\Gamma$ and $\Delta$ act on $\Free(\N^{\Gamma * \Delta})$ by restricting
  the left shift action of $\Gamma * \Delta$ to these subgroups.

\begin{lemma}\label{main_lemma}[Main Lemma]
  Let $\Gamma,\Delta$ be countable groups. 
  If $A \subset \Free(\N^{\Gamma * \Delta})$ is any Borel
  set, then at least one of the following holds:
  \begin{enumerate}
  \item There is a continuous injective function $f: \Free(\N^\Gamma) \to
  \Free(\N^{\Gamma * \Delta})$ that is
  equivariant with respect to the left shift action of $\Gamma$ on these spaces and
  such that $\ran(f) \subset A$. 

  \item There is an continuous injective function $f: \Free(\N^\Delta) \to
  \Free(\N^{\Gamma * \Delta})$ that is
  equivariant with respect to the left shift action of $\Delta$ on these spaces and
  such that $\ran(f) \subset \Free(\N^{\Gamma * \Delta}) \setminus A$.
  \end{enumerate}
\end{lemma}

\begin{proof}
  The main difficulty in our proof is arranging that our games 
  produce elements of $\Free(\N^{\Gamma * \Delta})$, and not merely elements
  of $\N^{\Gamma * \Delta}$. To
  begin, we make a definition that will get us halfway there.
  Let $Y$ be the largest invariant set of $y \in \N^{\Gamma * \Delta}$ 
  such that for all nonidentity
  $\gamma \in \Gamma$ and $\delta \in \Delta$, we have $\gamma \cdot y \neq y$, and 
  $\delta \cdot y \neq y $. That is, $Y$
  is the set of $x \in \N^{\Gamma * \Delta}$ such that for all $\alpha \in \Gamma * \Delta$, if $y =
  \alpha^{-1} \cdot x$, then $y$ has the property above. Note that $Y$
  contains $\Free(\N^{\Gamma * \Delta})$.

  Next, we give a definition that we will use to organize the turn on which
  $y(\alpha)$ is defined in our game for each  $\alpha \in \Gamma *
  \Delta$. Fix injective listings $\gamma_0, \gamma_1, \ldots$ and
  $\delta_0, \delta_1, \ldots$ of all the nonidentity elements of $\Gamma$ and
  $\Delta$ respectively. We define the turn function $t: \Gamma * \Delta
  \to \N$ as follows. First, define $t(e) = -1$. Then, for each nonidentity
  element $\alpha \in \Gamma
  * \Delta$, there is a unique sequence $i_0, i_1 \ldots i_m$ such that
  $\alpha = \gamma_{i_0} \delta_{i_1} \gamma_{i_2} \ldots$ or $\alpha =
  \delta_{i_0} \gamma_{i_1} \delta_{i_2} \ldots$. We define $t(\alpha)$ to
  be the least $n$ such that the associated sequence $i_0, i_1, \ldots i_m$
  for $\alpha$ has $i_j + j \leq n$ for all $j \leq m$. The key property of
  this definition is that if $i \leq n$, 
  and $\alpha$ is a $\Delta$-word or the
  identity, then $t(\gamma_i \alpha) \leq n$ if and only if $t(\alpha) < n$. Of
  course, this remains true when the roles of $\Gamma$ and $\Delta$ are
  switched. 

  Now given a Borel set $B \subset Y$, 
  and $k \in \N$ we define the following game $G^B_k$ for
  producing a $y \in \N^{\Gamma * \Delta}$ such that $y(e) = k$.
  Player I goes first, and the players alternate defining $y$ on finitely
  many nonidentity elements of $\Gamma * \Delta$ as follows. 
  On 
  the $n$th turn of the game for $n \geq 0$, player I must define $y(\alpha)$ on all
  $\Gamma$-words $\alpha$ with $t(\alpha) = n$, and then player II must
  respond by 
  defining $y(\alpha)$ on all $\Delta$-words $\alpha$ with $t(\alpha) =
  n$. 
  We give an illustration of how the game is played:
\[\begin{array}{ccc}
\text{I} && \text{II} \\
y(\gamma_0) && \\
&&  y(\delta_0) \\
y(\gamma_1) && \\
y(\gamma_0 \delta_0) && \\
y(\gamma_1 \delta_0) && \\
& \vdots &
\end{array}\]
  
  All that remains is to define the winning condition of the game. First,
  if the $y$ that is produced is in $Y$, then Player II wins the game if and
  only if $y$ is in $B$. If $y \notin Y$ then there must be some $\alpha$
  such that there is a nonidentity $\gamma \in \Gamma$ such that $\gamma
  \alpha^{-1} \cdot y = \alpha^{-1} \cdot y $,
  or there is a nonidentity $\delta \in \Delta$ such that $\delta 
  \alpha^{-1} \cdot y = \alpha^{-1} \cdot y$. In
  the former case, say $(\alpha,\Gamma)$ witnesses $y \notin Y$, and in the
  latter say $(\alpha,\Delta)$ witnesses $y \notin Y$. Say $\alpha$
  witnesses $y \notin Y$ if either $(\alpha,\Gamma$) or $(\alpha,\Delta)$
  witnesses $y \notin Y$. 
  Now if $(e,\Gamma)$
  witnesses $y \notin Y$, then player I loses. 
  Otherwise, if $(e,\Delta)$ witnesses $y \notin Y$ then
  player II loses. Finally, if neither of the above happens, then player I
  wins if and only if there is a $\Delta$-word $\alpha$ witnessing $y
  \notin Y$ such that for all $\Gamma$-words $\beta$ with $t(\beta) \leq
  t(\alpha)$, we have that $\beta$ does not witness $y \notin Y$. This
  finishes the definition of our game.

  Next, we associate to our set $A \subset \Free(\N^{\Gamma * \Delta})$ a
  set $B_A$ that we will use in the play of our game. Let $E_\Gamma$ be the
  equivalence relation on $Y$ where $x \E_\Gamma y$ if there is a $\gamma
  \in \Gamma$ such that $\gamma \cdot x = y$. Define $E_\Delta$ similarly.
  By Lemma~\ref{partition_non_independent} which we defer till later, we
  can find a Borel subset $C$ of $Y \setminus \Free(\N^{\Gamma * \Delta})$
  such that $C$ meets every $E_\Delta$-class on $Y \setminus
  \Free(\N^{\Gamma * \Delta})$ and its complement $\comp{C}$ meets every
  $E_\Gamma$-class on $Y \setminus \Free(\N^{\Gamma * \Delta})$. Let $B_A =
  A \union C$. Our use of $C$ here will be important at the end of the
  proof to ensure that we create a function into $\Free(\N^{\Gamma *
  \Delta})$ and not merely into $Y$.

  By Borel determinacy, either player I or player II has a winning strategy
  in $G^{B_A}_k$ for each $k \in \N$. So by the pigeon-hole principle,
  either player I wins $G^{B_A}_k$ for infinitely many $k$ or player II wins
  $G^{B_A}_k$ for infinitely many $k$. Assume the latter case holds, and
  let $S$ be the set of $k$ such that player II wins $G^{B_A}_k$. 
  An analogous argument will work in the case that player I
  wins for infinitely many $k$. Since there is a 
  continuous injective equivariant function from $\Free(\N^\Gamma)$ to
  $\Free(S^\Gamma)$, it will suffice to define a continuous injection $f:
  \Free(S^\Gamma) \to \Free(\N^{\Gamma * \Delta})$ that is equivariant
  with respect to the left shift action of $\Gamma$ on these spaces and
  such that $\ran(f) \subset A$. Fix winning strategies in each game
  $G^{B_A}_{k}$ for $k \in S$.

  We will define $f$ so that for all $x \in \Free(\N^{\Gamma})$ and all
  $\gamma \in \Gamma$, we have $f(x)(\gamma) = x(\gamma)$, and so that for
  all $x$, $f(x)$ will be a winning outcome of player II's winning strategy
  in the game $G^{B_A}_{x(e)}$. 

  We proceed as follows. Fix an $x$ in $\Free(\N^\Gamma)$. For each $\gamma
  \in \Gamma$ we will play an instance of the game
  $G^{B_A}_{x(\gamma^{-1})}$ whose outcome will be $\gamma \cdot f(x)$.
  We play these games for all $\gamma \in \Gamma$ simultaneously. The moves
  for player II in these games will be made by the winning strategies that
  we have fixed. We will specify how to move for player I in these games to
  satisfy our requirement that $f$ is equivariant and $f(x)(\gamma) =
  x(\gamma)$.

  So for each $\gamma \in \Gamma$, we are playing an instance of the game
  $G^{B_A}_{x(\gamma^{-1})}$ to define a $y \in \N^{\Gamma * \Delta}$ equal
  to $\gamma \cdot f(x)$. To begin, we have $\gamma \cdot f(x)(e) =
  x(\gamma^{-1})$ by the definition of the game.
  
  Inductively, suppose $\gamma \cdot f(x)(\alpha)$ is defined for all
  $\gamma \in \Gamma$ and all $\alpha$ with $t(\alpha) < n$. We need to
  make the $n$th move for player I in all our games. Suppose $\beta$ is a
  $\Gamma$-word with $t(\beta) = n$ so we can write $\beta = \gamma_i
  \alpha$ where $i \leq n$ and $t(\alpha) < n$. For all $\gamma \in
  \Gamma$, we now define $(\gamma \cdot f(x))(\gamma_i \alpha) =
  (\gamma_i^{-1} \cdot (\gamma  \cdot f(x)))(\alpha) = (\gamma_i^{-1} \gamma
  \cdot f(x))(\alpha)$, which has already been defined in the game
  associated to $\gamma_i^{-1} \gamma$ by assumption. Hence we can make the
  $n$th move for player I in all our games using this information. To
  finish the $n$th turn, the winning strategies for player II in these
  games respond with their $n$th moves, defining $\gamma \cdot
  f(x)(\delta_i \alpha)$ for all $i \leq n$ and all $\alpha$ such that
  $\alpha = e$ or $\alpha$ is a $\Delta$-word with $t(\alpha) < n$. 

  Based on our definition, it is clear that $f$ is injective, continuous,
  $\Gamma$-equivariant, and that $f(x)$ is an outcome of player II's
  winning strategy in $G_{x(e)}^{B_A}$. All that remains is to show
  $\ran(f) \subset A$. 

  First, we argue that for all $x \in \Free(\N^\Gamma)$, we have $f(x) \in
  Y$. Now since $x \in \Free(\N^\Gamma)$ and $f(x)(e) = x(e)$, we see
  that $(e,\Gamma)$ can not witness $f(x) \notin Y$. Further, since $f(x)$
  is a winning outcome of a strategy for player II, $(e,\Delta)$ can not
  witness $f(x) \notin Y$. Now we can prove inductively that $\alpha$ does
  not witness $f(x) \notin Y$ for all $x \in \Free(\N^\Gamma)$ and all $\alpha \in \Gamma * \Delta$
  with $t(\alpha) = n$. For each $n$ we do the case of $\Gamma$-words
  first, and then the case of $\Delta$-words. Suppose $\alpha$ is a
  $\Gamma$-word with $t(\alpha) = n$, so $\alpha = \gamma \beta$ for some
  $\gamma \in \Gamma$ and $\beta$ with $t(\beta) <
  t(\alpha)$. Since $\alpha^{-1} \cdot f(x) = \beta^{-1} \gamma^{-1} \cdot
  f(x) = \beta^{-1} \cdot f(\gamma^{-1} \cdot x)$ and $\beta$ does not
  witness 
  $f(\gamma^{-1} \cdot x) \notin Y$ by our induction
  hypothesis, we must have that $\alpha$ does not witness $f(x) \notin Y$.
  Now suppose $\alpha$ is a $\Delta$-word with $t(\alpha) = n$. We may
  assume no $\Gamma$-word $\beta$ with $t(\beta) \leq n$ witnesses $f(x)
  \notin Y$. Hence, we see that player II must ensure $\alpha$ does not
  witness $\alpha \cdot f(x) \notin Y$ otherwise they lose the game
  $G^{B_A}_{x(e)}$ used to define $f(x)$.
 
  For all $x$, since $f(x) \in Y$, we have $f(x) \in B_A$, since $f(x)$ is
  a winning outcome for player II in some $G^{B_A}_{x(e)}$. Finally, we
  claim that $f(x) \in A$ for all $x$. This is because $\ran(f)$ and $A$
  are $\Gamma$-invariant, $B_A = A \union C$, and $C$ does
  not contain any nonempty $\Gamma$-invariant sets by definition.
  \end{proof}

To finish establishing Lemma~\ref{main_lemma}, we must prove
Lemma~\ref{partition_non_independent} which was used to define the set $C$
above. We will prove a version for countably many equivalence
relations instead of merely two, since we will use this more general
version in a later paper. 

We begin by recalling a useful tool for organizing
constructions in Borel combinatorics.
Let $X$ be a standard Borel space. We let $[X]^{< \infty}$ denote the
standard Borel space of
finite subsets of $X$. If $E$ is a countable Borel equivalence relation,
we let $[E]^{< \infty}$ be the Borel subset of $[X]^{< \infty}$ consisting
of the $S \in [X]^{< \infty}$ such that $S$ is a subset of some
equivalence class of $E$. If $Y$ is a Borel subset of $[X]^{< \infty}$,
then the \define{intersection graph} on $Y$ is the graph $G$ where $R \G S$ for
distinct $R,S \in Y$ if $R \inters S \neq \emptyset$.

\begin{lemma}[{\cite[{Lemma 7.3}]{MR2095154}} {\cite[Proposition
2]{CM2}}]
  \label{intersection_graph_coloring}
  Suppose $E$ is a countable Borel equivalence relation and let $G$ be the
  intersection graph on $[E]^{<\infty}$. Then $G$ has a Borel $\N$-coloring.
\end{lemma}

We will often use this lemma in the following way. Suppose $E$ is a
countable Borel equivalence relation and $A$ is a Borel subset of
$[E]^{<\infty}$ containing at least one subset of every $E$-class. Then 
there is a Borel set $B \subset A$ such that elements of $B$ are pairwise
disjoint, and $B$ meets every $E$-class. 
To see this, pick some Borel $\N$-coloring of the intersection graph of $[E]^{<
\infty}$ using Lemma~\ref{intersection_graph_coloring}, and then let 
$B$ be the set of $R \in A$ that are assigned the least color of all
elements of $A$ from the same $E$-class.

We need a couple more definitions. Suppose that
$I \in \{1, 2, \ldots, \infty\}$ and $\{E_i\}_{i < I}$ are finitely many or countably many
equivalence relations on $X$. Then the $E_i$ are said to be
\define{non-independent} if there exists a sequence $x_0, x_1, \ldots,
x_n$ of distinct elements of $X$, and $i_0, i_1, \ldots i_n \in \N$
with $n \geq 2$ such that $i_j \neq i_{j+1}$ for $j < n$, $i_n \neq i_0$, and $x_0 \E_{i_0}
x_1 \E_{i_1} x_2 \ldots x_n \E_{i_n} x_0$. We say this pair of sequences
$x_0, \ldots, x_n$ and $i_0, \ldots, i_n$ witnesses the non-independence of
the $E_i$. The $E_i$ are said to be \define{independent} if they are not
non-independent.
The \define{join of the $E_i$}, denoted $\bigvee_{i < I} E_i$, is
the smallest equivalence relation containing all the $E_i$. Precisely, $x$
and $y$ are $\bigvee_{i < I} E_i$-related if there is a sequence $x_0, x_1,
\ldots x_n$ of elements in $X$ such that $x = x_0$, $y = x_n$, and for all
$j < n$, we have $x_j \E_i x_{j+1}$ for some $i < I$. Finally, we say that
the
$E_i$ are \define{everywhere non-independent} if for every 
$\bigvee_{i < I} E_i$ equivalence class $A \subset X$, the restrictions of
the $E_i$ to $A$ are not independent.

\begin{lemma}\label{partition_non_independent}
Suppose that $I \in \{1, 2, \ldots, \infty\}$ and $\{E_i\}_{i < I}$ are
countable Borel equivalence relations on a standard Borel space $X$ that
are everywhere non-independent. Then there exists a Borel partition
$\{A_i\}_{i < I}$ of $X$ such that for all $i < I$, 
$\comp{A_i}$ meets every $E_i$-class.
\end{lemma}
\begin{proof}
Using Lemma~\ref{intersection_graph_coloring}, let $C \subset [\bigvee_{i < I} E_i]^{<
\infty}$ be a Borel set
containing at least one subset of each equivalence class of $\bigvee_{i <
I} E_i$ such that the elements of $C$ are pairwise disjoint and each
set $\{x_0, x_1, \ldots, x_n\} \in C$ can be assigned an order $x_0, x_1,
\ldots, x_n$ and an associated sequence $i_0, \ldots, i_n$ of natural
numbers such that these sequences witness the failure of the independence
of the $E_i$. Fix a Borel way of assigning such an order and associated
$i_0, \ldots, i_n$ to each element of $C$. Define disjoint sets $\{A_{i,0}\}_{i < I}$
by setting $x \in A_{i,0}$ if there is an element $\{x_0, \ldots, x_n\}$ of $C$ with
associated sequence $i_0, \ldots, i_n$ and a $j\leq n$ such that 
$x = x_j$ and $i = i_j$. 
Note that for all $i < I$ and for all $x \in A_{i,0}$, there is a $y \in
[x]_{E_i}$ and $j \neq i$ such that $y \in A_{j,0}$. 

Let $k_0, k_1, \ldots$ be a sequence containing each number less than $I$
infinitely many times. Given $\{A_{i,n}\}_{i < I}$ we construct disjoint
sets $\{A_{i,n+1}\}_{i < I}$, where $A_{i,n+1} \supset A_{i,n}$ as follows.
Let $B_{i,n+1}$ be the set of $x$ such that 
$x \in [A_{k_n,n}]_{E_i}$ and for all $j < i$, $x \notin
[A_{k_n,n}]_{E_j}$. Then let $A_{i,n+1} =
A_{i,n} \union (B_{i,n+1} \setminus \union_{j \neq i} A_{j,n})$.
The sets $A_{i,n}$ are Borel, since
every
$E_i$ is generated by the Borel action of a countable group by the
Feldman-Moore theorem~{\cite[Theorem 1.3]{MR2095154}}.

It is easy to prove by induction that for all $x \in A_{i,n}$, there is a
$y \in [x]_{E_i}$ and a $j \neq i$ such that $y \in A_{j,n}$. Let $A_i =
\union_n A_{i,n}$, which are disjoint and partition the space. 
\end{proof}

\section{Applications to Borel chromatic numbers and matchings}
\label{sec:applications}

We now show how our main lemma can be applied to prove the theorems 
discussed in the introduction. Recall that if $G$ and $H$ are Borel graphs
on the standard Borel spaces $X$ and $Y$ respectively, then a \define{Borel
homomorphism} from $G$ to $H$ is a Borel function $f: X \to Y$ such that $x
\G y$ implies $f(x) \mathrel{H} f(y)$. It is clear that if there is a Borel
homomorphism $f$ from $G$ to $H$, then $\chi_B(G) \leq \chi_B(H)$; if $c$
is a Borel coloring of $H$, then $c \circ f$ is a Borel coloring of $G$. 

\begin{thm}\label{chromnum_free_product}
  If $\Gamma$ and $\Delta$ are finitely generated marked groups, then
  \[\chi_B(G(\Gamma*\Delta,\N)) \geq \chi_B(G(\Gamma,\N)) +
  \chi_B(G(\Delta,\N)) -1\]
\end{thm}
\begin{proof}
  Suppose $\chi_B(G(\Gamma,\N)) = n+1$ and $\chi_B(G(\Delta,\N)) = m+1$ so
  that $G(\Gamma,\N)$ has no Borel $n$-coloring and $G(\Delta,\N)$ has no
  Borel $m$-coloring. Now suppose $c: \Free(\N^{\Gamma * \Delta}) \to
  \{0,1,
  \ldots, (n+m-1)\}$ was a Borel $n+m$-coloring of
  $G(\Gamma * \Delta,\N)$ and let $A$ be the set of $x$ such that $c(x) <
  n$. If $f$ is 
  the continuous equivariant function produced by
  Lemma~\ref{main_lemma}, then $c \circ f$ gives either a 
  Borel $n$-coloring of
  $G(\Gamma,\N)$ or a Borel $m$-coloring of $G(\Delta,\N)$, both of which
  are contradictions. 
\end{proof}

Let $\mathcal{C}$ be the class of finitely generated marked groups $\Gamma$ such
that $G(\Gamma,\N)$ is $n$-regular, and $\chi_B(G(\Gamma,\N))
= n+1$, so that the upper bound on the Borel chromatic number of
$G(\Gamma,\N)$ given by
Theorem~\ref{KST_ncoloring} is sharp. Brooks's
theorem in finite graph theory (see e.g. {\cite[Theorem 5.2.4]{MR2159259}})
implies that the finite
groups included in $\mathcal{C}$ are exactly those whose Cayley graphs are odd cycles
or complete graphs on $n$ vertices. The only prior results giving
infinite groups in $\mathcal{C}$ are from \cite{MR1667145} where we have
that $\Z$
and $\Z/2\Z * \Z /2\Z$ are in $\mathcal{C}$ when equipped with their usual
generators. Conley and Kechris~{\cite[Theorem 0.10]{MR3019078}} have shown that
these are the only two groups with finitely many ends that are in
$\mathcal{C}$.
Theorem~\ref{chromnum_free_product} implies that $\mathcal{C}$
is closed under free products; if $G(\Gamma,\N)$ is $n$-regular and
$G(\Delta,\N)$ is $m$-regular, then $G(\Gamma * \Delta,\N)$ is
$n+m$-regular. For example, $\chi_B(G((\Z/2\Z)^{*n},\N)) = n+1$, and
$\chi_B(G(\F_n,\N)) = 2n+1$ for all $n$.

Next, we will show that the Borel chromatic number of an $n$-regular
acyclic Borel graph can take any of the possible values between $2$ and
$n+1$ allowed by Theorem~\ref{KST_ncoloring}. 

We begin with an easy lemma.
\begin{lemma}\label{injective_acyclic_homomorphism}
  Suppose $G$ and $H$ are acyclic Borel graphs on the standard Borel spaces $X$ and $Y$,
  where $\chi_B(G) \geq 2$. Suppose also $f \from X \to Y$ is an injective
  Borel homomorphism from $G$ to $H$ such that $\forall x, y \in X$, if $x$
  and $y$ are in different connected components of $G$, then $f(x)$ and
  $f(y)$ are in different connected components of $H$. Then if $A =
  [\ran(f)]_H$ is the saturation of the range of $f$ under the
  connectedness relation of $H$, then $\chi_B(G) = \chi_B(H \restriction A)$. 
\end{lemma}
\begin{proof}
  $\chi_B(G) \leq \chi_B(H \restriction A)$ since
  there is a Borel homomorphism from $G$ to $H \restriction A$.
  It remains to show that $\chi_B(H \restriction A) \leq
  \chi_B(G)$. Suppose $c \from X \to Z$ is a Borel coloring of $G$. Fix two
  colors $z_0$ and $z_1 \in Z$. Now we construct a Borel coloring $c' \from
  A \to Z$ of $H \restriction A$ as follows. If $y \in
  \ran(f)$, then let $c'(y) = c(f^{-1}(y))$. Otherwise, there is a unique
  path in $H$ of shortest length $l$ from $y$ to an element $y' \in
  \ran(f)$. If $c(f^{-1}(y')) = z_0$, then let $c(y) = z_1$ if $l$ is odd
  and $c(y) = z_0$ if $l$ is even. If $c(f^{-1}(y')) \neq z_0$,
  then let $c(y) = z_0$ if $l$ is odd and $c(y) = z_1$ if $l$ is even.
\end{proof}

We are ready to proceed.

\begin{thm}\label{possible_chromnum}
  For every $n \geq 1$ and every $m \in \{2, \ldots, n+1\}$, there is a
  $n$-regular acyclic Borel graph $G$ with $\chi_B(G) = m$.
\end{thm}
\begin{proof}
  We have shown that for every $k \geq 2$ we have $\chi_B(G(\Zt^{*k},\N))
  = k+1$. Given $m \in \{2, \ldots, n+1\}$, canonically identify $\Zt^{*(m-1)}$
  with a subgroup of $\Zt^{*n}$. Now let 
  $f \from 
   \Free(\N^{\Zt^{*(m-1)}}) \to \Free(\N^{\Zt^{*n}})$ be the function where 
  \[ f(x)(\gamma) = \begin{cases} x(\gamma) & \text{ if $\gamma \in
  \Zt^{*(m-1)}$} \\
  0 & \text{ otherwise.}
  \end{cases}\]
  We finish by applying Lemma~\ref{injective_acyclic_homomorphism} with
  $G = G(\Zt^{*(m-1)},\N)$, $H = G(\Zt^{*n},\N)$ and $f$ as above to obtain a
  Borel set $A$ saturated under the connectedness relation of $H$ so $H
  \restriction A$ is $n$-regular and $\chi_B(H \restriction A) = m$. 
\end{proof}

The only case we know of where Theorem~\ref{chromnum_free_product} gives a
sharp lower bound for the chromatic number of $G(\Gamma * \Delta,\N)$ is
when $\Gamma$ and $\Delta$ are in the class $\mathcal{C}$ we have
discussed above. However, it is open whether the lower bound of
Theorem~\ref{chromnum_free_product} can ever be exceeded. 

\begin{question}\label{exceed}
  Are there finitely generated marked groups $\Gamma$ and $\Delta$ such
  that $\chi_B(G(\Gamma
  * \Delta, \N)) > \chi_B(G(\Gamma,\N)) + \chi_B(G(\Delta,\N)) - 1$?
\end{question}

Now there is another obvious upper bound on the Borel chromatic number of 
$G(\Gamma * \Delta,\N)$ which is better in some cases than that of
Theorem~\ref{KST_ncoloring}:

\begin{prop}\label{chrom_prod} If $\Gamma$ and $\Delta$ are finitely
generated marked groups, then $\chi_B(G(\Gamma * \Delta,\N)) \leq \chi_B(G(\Gamma,\N))
\chi_B(G(\Delta,\N))$
\end{prop}
\begin{proof}
We can decompose $G(\Gamma * \Delta,X)$ as the
disjoint union of two Borel graphs $G_\Gamma$ and $G_\Delta$ given by the
edges corresponding to generators of $\Gamma$ and $\Delta$ respectively.
Since $G_\Gamma$ and $G_\Delta$ are induced by free actions of $\Gamma$ and
$\Delta$, their Borel chromatic numbers are less than or equal to
$\chi_B(G(\Gamma,\N))$ and
$\chi_B(G(\Delta,\N))$ and hence we can use pairs of these colors to color
\mbox{$G(\Gamma * \Delta,\N)$}. 
\end{proof}

It is likewise open whether this upper bound can ever be achieved.

\begin{question}\label{product_question}
  Are there nontrivial finitely
  generated marked groups $\Gamma$ and $\Delta$ such that
  $\chi_B(G(\Gamma * \Delta, \N)) =  \chi_B(G(\Gamma,\N))
  \chi_B(G(\Delta,\N))$?
\end{question}

A positive answer to this question would also give a positive answer to
Question~\ref{exceed}. It seems natural to believe
Question~\ref{product_question} has a positive answer in cases where
$G(\Gamma * \Delta,\N)$ is $n$-regular and $\chi_B(G(\Gamma,\N))
\chi_B(G(\Delta,\N)) \leq n$, so that the bound of
Proposition~\ref{chrom_prod} is better than that of
Theorem~\ref{KST_ncoloring}. For example, if $m > 2$ is even, and we
generate $\Z/m\Z$ by a single element, then $G(\Z/m\Z * \Z/m\Z,\N)$ is
$4$-regular and has Borel chromatic number $\leq 4$ by
Proposition~\ref{chrom_prod}. Likewise, $\Z^n$ is another source of such
examples, since $G(\Z^n,\N)$ is a $2n$-regular Borel graph with
$\chi_B(G(\Z^n,\N)) \leq 4$ by~\cite{GaoJackson}. 

Next, we turn to matchings and edge colorings.  
We begin with the 
following theorem on disjoint complete sections. 

\begin{thm} \label{nodcs}
  Let $\Gamma$ and $\Delta$ be countable groups. Let $E_\Gamma$ be the
  equivalence relation on $\Free(\N^{\Gamma * \Delta})$ where $x
  \E_\Gamma y$ if there exists a $\gamma \in \Gamma$ such that
  \mbox{$\gamma \cdot x = y$}.
  Define $E_\Delta$ analogously. Then $E_\Gamma$ and $E_\Delta$ do not have
  Borel disjoint complete sections.
\end{thm}
\begin{proof}
  Let $A$ be any Borel subset of $\Free(\N^{\Gamma *
  \Delta})$. Then the range of the $f$ produced by
  Lemma~\ref{main_lemma} is 
  either an $E_\Gamma$-invariant set contained in $A$, or an
  $E_\Delta$-invariant set contained in the complement of $A$. Hence, $A$
  cannot simultaneously meet every $E_\Delta$ class and have its complement
  meet every $E_\Gamma$-class.
\end{proof}

We now use this fact to obtain a couple of results on matchings and
edge colorings of Borel bipartite graphs.

\begin{thm}\label{matchings_a}
  For every $n > 1$, there exists an $n$-regular acyclic Borel bipartite
  graph with no Borel perfect matching.
\end{thm}
\begin{proof}
  Let $\Gamma = \Delta = \Z/n\Z$ in Theorem~\ref{nodcs}. Let $Y \subset
  [\Free(\N^{\Gamma*\Delta})]^n$ be the standard Borel space consisting of the
  equivalence classes of $E_\Gamma$ and $E_\Delta$. Let $G$ be the
  intersection graph on $Y$. This is an
  $n$-regular acyclic Borel bipartite graph. If $M \subset Y \times Y$ was a Borel perfect matching for
  $G$, then setting \[A = \{x \in \N^{\Gamma * \Delta} : \exists (R,S) \in M \text{ such that
  $\{x\} = R \inters S$}\},\] we see that $A$ and the complement of $A$
  would be Borel disjoint complete sections for $E_\Gamma$ and $E_\Delta$,
  contradicting Theorem~\ref{nodcs}.
\end{proof}

The graph used above was suggested as a candidate for
a graph with no perfect matching by Conley and Kechris \cite{MR3019078}. Lyons and
Nazarov \cite{MR2825538} have shown that in the case $n = 3$, this graph
has a measurable matching with respect to a natural measure.

\begin{thm}\label{maximal_edge}
  For every $n$, there exists an $n$-regular acyclic Borel bipartite graph
  with no Borel edge coloring with $2n-2$ colors.
\end{thm}
\begin{proof}
  We use the same graph as in Theorem~\ref{matchings_a}. Suppose for a
  contradiction that it had a Borel edge coloring with $2n-2$ colors. By
  the pigeonhole principle, each vertex of $G$ must be incident to at least
  one edge assigned an even color, and at least one edge assigned an odd
  color. Let $A$ be the set of points $x$ in $\Free(\N^{\Gamma *
  \Delta})$ such that $\{x\} = R \inters S$ where $R$ is an equivalence
  class of $E_\Gamma$, $S$ is an equivalence class of $E_\Delta$, and the
  edge $(R,S)$ in $G$ is colored with an even color. Then $A$ is a complete
  section for $E_\Gamma$, and the complement of $A$ is a complete section
  for $E_\Delta$, contradicting Theorem~\ref{nodcs}.
\end{proof}

Now we can give an exact characterization of the possible Borel edge chromatic
numbers of $n$-regular acyclic Borel bipartite graphs.

\begin{thm}\label{edge_coloring_a}
  For every $n \geq 1$ and every $m \in \{n, \ldots, 2n-1\}$, there is an 
  $n$-regular acyclic Borel bipartite graph $G$ such that
  $\chi'_B(G) = m$.
\end{thm}
\begin{proof}
  Let $G$ be an $n$-regular acyclic Borel graph on $X$ with $\chi'(G) \geq
  2n-1$ by Theorem~\ref{maximal_edge}. Let $m$ be an element of $\{n,
  \ldots, 2n-1\}$. There is an edge coloring of $G$ using $2n-1$ colors by
  Theorem~\ref{KST_ncoloring} and our discussion in the introduction before
  Theorem~\ref{edge_coloring}. Let $G' \subset G$ be the set of
  edges colored using one of the first $m$ colors. Then clearly the graph
  $G'$ on $X$ has a Borel edge coloring with $m$ colors. It cannot have a
  Borel edge coloring with $m-1$ colors as this would give an edge coloring
  of $G$ with $2n-2$ colors. Now let $Y$ be an uncountable standard Borel
  space, and let $H$ be an extension of $G'$ to an $n$-regular Borel
  bipartite graph $H$ on $X \disjointunion Y$ such that each connected
  component of $H \setminus G'$ has at most one point in $X$.
  Then $\chi'_B(H) = m$.
\end{proof}

In the theorems we have proved above, we have mostly worked on spaces of
the form $\Free(\N^\Gamma)$. As we described in the introduction, this is
quite natural since the graph $G(\Gamma,\N)$ achieves the maximal chromatic
number of all Borel graphs generated by free actions of $\Gamma$. 
However, it is interesting
to ask what happens when we change our base space to be finite. For
example, it is an open question whether there is a dichotomy characterizing when a pair of
countable Borel equivalence relations admits Borel disjoint complete sections,
and here we would like to know 
whether Theorem~\ref{nodcs} remains true when we change $\N$ to be
some finite $k$. As we will see, this is the case when $k = 3$, but it is open
for $k = 2$.
Likewise, we would like to compute the Borel chromatic number of graphs of
the form $G(\Gamma,k)$ for $k \geq 2$. Clearly, if $k \leq m$ are both at
least $2$, then $\chi_B(G(\Gamma,k)) \leq
\chi_B(G(\Gamma,m)) \leq \chi_B(G(\Gamma,\N))$. 
It is open whether these chromatic numbers can ever be
different\footnote{Recently, Seward and Tucker-Drob~\cite{1402.4184} have answered this
question in the negative. They show that for every countable group
$\Gamma$, there is an equivariant Borel function from $\Free(\N^{\Gamma})
\to \Free(2^{\Gamma})$.} :

\begin{question}\label{base_space}
  Does there exist a finitely generated marked group $\Gamma$ such that $\chi_B(G(\Gamma,\N))
  \neq \chi_B(G(\Gamma,2))$?
\end{question}

Certainly, there are no obvious tools to show such chromatic numbers can be
different. One approach to showing that these chromatic numbers are
always the same would be to show the existence of a Borel homomorphism from
$G(\Gamma,\N)$ to $G(\Gamma,2)$. To do this it would be sufficient to find
an equivariant Borel function from $\Free(\N^\Gamma)$ to $\Free(2^\Gamma)$.
We note that such a function could not be injective in the case when
$\Gamma$ is sofic (which includes all the examples of groups we have
discussed). This follows from results of Bowen on sofic entropy, as pointed
out by Thomas~{\cite[Theorem 6.11]{MR2914864}}. 

In the measurable context, when $(X,\mu)$ is a standard probability space, we can say a bit more 
about the $\mu^\Gamma$-measurable chromatic number of graphs of the form
$G(\Gamma,X)$, as 
$X$ and $\mu$ vary. Recall from
\cite{MR3019078} that the \define{$\mu$-measurable chromatic number} of a
Borel graph $G$ on a standard probability space $(X,\mu)$ is the least
cardinality of a Polish space $Y$ such that there is a $\mu$-measurable
coloring $c: X \to Y$ of $G$. Now given Borel actions $a$ and $b$ of $\Gamma$ on the
Borel probability spaces $(X,\mu)$ and $(Y,\nu)$ respectively, a \define{factor map}
from $a$ to $b$ is a $\mu$-measurable equivariant function $f: X \to
Y$ such that the pushforward of $\mu$ under $f$ is $\nu$.
Bowen~{\cite[Theorem 1.1]{MR2811154}} has shown that if $\Gamma$ contains a
nonabelian free subgroup, then given any nontrivial probability measures
$\mu$ and $\nu$ on the standard Borel spaces $X$ and $Y$, there is a factor map
from the left shift action of $\Gamma$ on $(X^\Gamma,\mu^\Gamma)$ to the
left shift action of $\Gamma$ on $(Y^\Gamma,\nu^\Gamma)$. Hence, the
$\mu^\Gamma$-measurable chromatic number of $G(\Gamma,X)$ is equal to the
$\nu^\Gamma$-measurable chromatic number of $G(\Gamma,Y)$ for all such
$(X,\mu)$ and $(Y,\nu)$. For some more results of this type for nonamenable
groups in general, see \cite{MR2931910}.
 
We now return to the pure Borel context, and end this section by noting that we
have the following variant of Lemma~\ref{main_lemma} for finite base
spaces. This lemma can be proved using a nearly identical argument to that
of Lemma~\ref{main_lemma} except changing the application of the
pigeon-hole principle in the obvious way. From this, one can derive
versions of all of the Theorems above for finite base spaces. For example,
we have $\chi_B(G(\Gamma*\Delta,m+n-1)) \geq \chi_B(G(\Gamma,m)) +
\chi_B(G(\Delta,n)) -1$, and Theorem~\ref{nodcs} and its corollaries hold
using $3$ instead of $\N$.

\begin{lemma}\label{main_lemma_fb}
  Let $\Gamma,\Delta$ be countable groups and $m,n \geq 2$ be finite.
  If $A \subset \Free((m+n-1)^{\Gamma * \Delta})$ is any Borel
  set, then at least one of the following holds:
  \begin{enumerate}
  \item There is an continuous injective function $f: \Free(n^\Gamma) \to
  \Free((m+n-1)^{\Gamma * \Delta})$ that is
  equivariant with respect to the left shift action of $\Gamma$ on these
spaces and
  such that $\ran(f) \subset A$.

  \item There is an continuous injective function $f: \Free(m^\Delta) \to
  \Free((m+n-1)^{\Gamma * \Delta})$ that is
  equivariant with respect to the left shift action of $\Delta$ on these
spaces and
  such that $\ran(f) \subset \comp{A}$.
  \end{enumerate}

\end{lemma}

\section{Disjoint complete sections for measure and category}
\label{contrast}

We turn now to the question of whether Theorem~\ref{nodcs} can be proved
using purely measure theory or Baire category. In the case when $\Gamma =
\Delta = \Z/2\Z$, we can prove Theorem~\ref{nodcs} using either of these
two tools. If the generators of $\Gamma$ and $\Delta$ are $\alpha$ and
$\beta$, then any nontrivial product probability measure on $\Free(\N^{\Zt
* \Zt})$ has the property that the map $x \mapsto \alpha \beta \cdot x$ is
ergodic, and the two maps $x \mapsto \alpha \cdot x$ and $x \mapsto \beta
\cdot x$ are both measure preserving. This is enough to conclude the
Theorem~\ref{nodcs} in this case. We can similarly give a Baire category
argument using generic ergodicity. We will show that $\Gamma = \Delta =
\Z/2\Z$ is the only nontrivial pair of $\Gamma$ and $\Delta$ for which
measure or category can prove Theorem~\ref{nodcs}. 

We begin by showing that Borel disjoint complete sections exist in the
measure context for aperiodic countable Borel equivalence relations. Recall
that an equivalence relation is said to be \define{aperiodic} if all of its
equivalence classes are infinite. 

\begin{lemma}\label{lemma:measure_dcs}
  Let $\mu$ be a Borel probability measure on a standard Borel space $X$. Then if
  $E$ and $F$ are aperiodic countable Borel equivalence relations on $X$,
  there exist disjoint Borel sets $A$ and $B$ such that $A$ meets
  $\mu$-a.e. equivalence class of $E$ and $B$ meets $\mu$-a.e. equivalence
  class of $F$. 
\end{lemma}
\begin{proof}
  It follows from the marker lemma~{\cite[Lemma 6.7]{MR2095154}} that we can
  find a decreasing sequence $C_0 \supset C_1 \supset \ldots$ of Borel sets
  that are each complete sections for both $E$ and $F$ and such that their
  intersection $\biginters C_i$ is empty.

  Note that for each $n$ and $\epsilon > 0$, there is $i > n$ such that
  \[\mu([C_n \setminus C_i]_E) > 1 - \epsilon \text{ and } \mu([C_n
  \setminus C_i]_F) > 1 - \epsilon \]
  It follows that we can find a strictly increasing sequence $(i_k)_k$ such
  that 
  \[\mu([ \bigunion_{k \geq 0} (C_{i_{2k}} \setminus C_{i_{2k+1}})]_E) = 1
  \text{ and } \mu([ \bigunion_{k \geq 0} (C_{i_{2k+1}} \setminus
  C_{i_{2k+2}})]_F) = 1 \]
  Now set $A = \union_{k \geq 0} (C_{i_{2k}} \setminus C_{i_{2k+1}})$ and
  $B = \union_{k \geq 0} (C_{i_{2k+1}} \setminus C_{i_{2k+2}})$. 
\end{proof}

Our goal is to extend this result to all pairs of countable Borel
equivalence relations $E$ and $F$ where every $E$ class has at least $2$
elements and every $F$-class has at least $3$ elements. We will do this by
reducing it to the case we have already proved above. More precisely, in
Theorem~\ref{dcs_equiv_1} we will show that several types of problems are
equivalent in a Borel way to the problem of finding Borel disjoint
complete sections for pairs of such equivalence relations. That is, to each
instance of each type of problem, we will demonstrate
how to construct an instance of each of the other types so that a
solution to these problems can be transformed
in a Borel way into a solution of the original problem. The exact sense
in which this is done will be clear in our proof. 
Of course, the idea of reductions between combinatorial
problems has a long history. For an example of recent work with a similar
effective flavor, see \cite{1212.0157}.

We first introduce another combinatorial problem. 
If $G$ is a graph on $X$, an \define{antimatching} of $G$ is a function
\mbox{$f:X \to X$} such that for all $x \in X$, we have $x \mathrel{G}
f(x)$ and $f(f(x)) \neq x$. A \define{partial antimatching} of $G$ is a
partial function \mbox{$f: X \to X$} satisfying these conditions for all $x \in
\dom(f)$.

We have the following lemma constructing antimatchings in the topological
context, using a result of Conley and Miller on the existence of Borel
matchings in the topological context: 

\begin{lemma} \label{antimatching_cat}
  Suppose $n \geq 3$ and $G$ is an acyclic Borel bipartite $n$-regular
  graph on a Polish space $X$. Then there exists a Borel antimatching of
  $G$ modulo a $G$-invariant meager set. 
\end{lemma}
\begin{proof}
  By~\cite{ConleyMiller}, there exists a Borel perfect matching for $G$
  restricted to a $G$-invariant meager set $C$. Let $A$ be one half of a Borel
  partition of $X$ witnessing the bipartiteness of $G$, and let $M$ be the
  Borel perfect matching of $G \restriction C$. Then we can construct a Borel antimatching
  $f$ for $G \restriction C$ in the following way: if $x \in A$ and $\{x,y\} \in M$, then
  set $f(x) = y$. If $x \notin A$, then choose some neighbor $y$ of $x$
  such that $\{x,y\} \notin M$ and set $f(x) = y$.
\end{proof}

The following lemma is useful when dealing with Borel antimatchings.

\begin{lemma}\label{extension_of_antimatchings}
  Suppose $G$ is a locally countable Borel graph, and $f$ is
  a partial Borel antimatching of $G$ such that $\ran(f) \subset \dom(f)$, 
  and every connected component of $G$ contains some $x \in
  \dom(f)$. Then $f$ can be extended to a total Borel antimatching $f^*$ of
  $G$.
\end{lemma}
\begin{proof}
  Define $f^*$ as follows. Let $f^*(x) = f(x)$ if $x \in \dom(f)$.
  Otherwise, let $f^*(x) = y$, for some neighbor $y$ of $x$ such that
  the distance in $G$ from $y$ to an element of $\dom(f)$ is as small as possible
  (using 
  Lusin-Novikov
uniformization~{\cite[18.10, 18.15]{MR1321597}}
  to choose such a $y$ when there is more than one). Then clearly $f^*(f^*(x)) \neq x$
  since for any $x \notin \dom(f)$, we have that $f^*(x)$ is closer to some
  element of $\dom(f)$ than $x$.
\end{proof}

Throughout this section, we assume that we have a Borel linear order on all
our standard Borel spaces. Thus, when we speak of the least element of some
finite subset of a standard Borel space, we are referring to the least
element with respect to this order. One way of obtaining such a linear
order is via a Borel bijection with a standard Borel space equipped with a
canonical Borel linear ordering, such as the one on $\R$. These linear orderings are useful when we need to break ``ties''
in our constructions when we are faced with some irrelevant choice. In cases
where we need to choose one of finitely many points, we will generally
break ties by choosing the least point according to this ordering. In cases where
we need to choose one of countably many options, we can use 
uniformization as we have above.

\begin{lemma}\label{antimatching_2reg}
  If $G$ is an acyclic locally finite Borel graph of degree $\geq 2$,
  then there is a partial Borel antimatching $f$ of $G$ such that $G
  \restriction
  \comp{(\dom(f))}$ is $2$-regular.
\end{lemma}
\begin{proof}
  Let $G$ be a locally finite Borel graph of degree $\geq 2$ on a
  standard Borel space $X$. Using Lemma~\ref{intersection_graph_coloring}, 
  let $\{A_i\}_{i \in \N}$ be a Borel partition of $X$ 
  such that for all $i$, for all distinct $x, y \in A_i$, 
  the distance between $x$ and $y$ in $G$ is greater than $2$. 

  Let $k_0, k_1, \ldots$ be a sequence
  containing each natural number infinitely many times.
  We define a sequence $f_0 \subset f_1 \subset \ldots$
  of partial Borel antimatchings whose union will be the $f$ we desire. 
  These $f_i$ will all have the property that if
  $x \in \ran(f_i)$ and $x \notin \dom(f_i)$, then there
  exist exactly two neighbors $y$ of $x$ such that $y
  \notin \dom(f_i)$ or $f_i(y) \neq x$. 

  Let $f_0 = \emptyset$. Now we define $f_{i+1} \supset f_{i}$. For each $x
  \in A_{k_{i}}$ such that $x \notin \dom(f_i)$, do the following: if there
  exists some neighbor $y$ of $x$ such that $y \in \dom(f_i)$ and $f_i(y)
  \neq x$, then use uniformization to choose some such $y$ and define 
  $f_{i+1}(x) = y$. If there does not exist any such $y$ and $x \notin
  \ran(f_i)$, then 
  choose exactly two neighbors $y_1$ and $y_2$
  of $x$ and define $f_{i+1}(y) = x$ for all neighbors $y$ of $x$
  that are not equal to $y_1$ or $y_2$.

  Let $f = \bigunion_{i \in \N} f_i$. Now if 
  $x \notin \dom(f)$, there are exactly two neighbors $y$ of
  $x$ such that $y \notin \dom(f)$ or $f(y) \neq x$. However, if $x$ had a
  neighbor $y$ such that $f(y) \neq x$, then we would have $x \in \dom(f)$.
  Hence, both these two $y$ must not be in $\dom(f)$. Thus, 
   $G \restriction \comp{(\dom(f))}$ is $2$-regular. 
\end{proof}

We are now ready to proceed.

\begin{thm}\label{dcs_equiv_1}
  Suppose $n \geq 3$. Then the following statements are all false. However,
  the statements are all true modulo a nullset with respect to any
  Borel probability measure, and true modulo a meager set with respect to
  any compatible Polish topology. 
  \begin{enumerate}
    \item Every pair $E$ and $F$ of countable Borel equivalence relations 
    on a standard Borel space $X$ 
    such that the $E$-classes all have cardinality $\geq 3$ and the $F$-classes
    all have cardinality $\geq 2$ admits disjoint Borel complete sections. 
    \item Every pair $E$ and $F$ of independent aperiodic countable Borel
    equivalence relations admits 
    disjoint Borel complete sections. 
    \item Every locally finite Borel graph $G$ having degree at least
    $3$ has a Borel antimatching. 
    \item Every acyclic Borel bipartite $n$-regular graph $G$ has a Borel
    antimatching. 
  \end{enumerate}
\end{thm}

\begin{proof}
  (1) is false by Theorem~\ref{nodcs}. (2) is true in the measure-theoretic
  context by Lemma~\ref{lemma:measure_dcs}. (4) is true in the topological
  context by Lemma~\ref{antimatching_cat}. 
  
  We will finish the proof of the theorem by showing (1) $\implies$ (2)
  $\implies$ (3) $\implies$ (1), and (3) $\implies$ (4) $\implies$ (2).
  Further, each of these implications will be done in a ``local'' way so
  that these implications also yield the truth of these statements in the
  measure and category contexts. We will discuss this more in
  what follows. 

  (1) $\implies$ (2) is obvious.

  (2) $\implies$ (3). Let $X$ be a standard Borel space. We will begin by proving
  the special case where $G$ is a $3$-regular acyclic Borel graph on
  $\{0,1\} \times X$ where $(0,x) \G (1,y)$ if and only if $x = y$. For $i \in \{0,1\}$,
  let $F_i$ be the equivalence relation on $X$ such that $x \mathrel{F_i} y$ if and only if $(i,x)$
  and $(i,y)$ are in the same connected component of $G \restriction \{i\}
  \times X$. The $F_i$ are
  independent because $G$ is acyclic. Let $B \subset X$ be a Borel set
  such that $B$ is a complete section for $F_0$ and $\comp{B}$ is a complete
  section for $F_1$. We can use $B$ to define a Borel antimatching. The rough
  idea is to direct elements of $\{0\} \times X$ towards elements of
  $B$ and direct elements of $\{1\} \times X$ away from elements of
  $B$. 
  
  If $x \in B$, define $f((0,x)) = (1,x)$. Then let $z$ be a point
  of $\comp{B}$ such that $(1,z)$ is closest to $(1,x)$ in $G \restriction \{1\}
  \times X$ (breaking ties as usual), and define $f((1,x)) = (1,y)$ where $(1,y)$ is the neighbor of
  $(1,x)$ along the path from $(1,x)$ to $(1,z)$. Likewise, if $x \in
  \comp{B}$, define $f((1,x)) = (0,x)$, let $z$ be a point of $B$
  such that $(0,z)$ is closest to $(0,x)$ in $G \restriction \{0\} \times X$,
  and define $f((0,x)) = (0,y)$ where $(0,y)$ is the neighbor of $(0,x)$
  along the path from $(0,x)$ to $(0,z)$.

  Now let $G$ be an arbitrary locally finite Borel graph on $X$ having degree at least
  $3$. 
  First, we may assume that $G$ is acyclic. To see this, use
  Lemma~\ref{intersection_graph_coloring} to obtain a Borel set $C$ of
  pairwise disjoint cycles that contains at least one cycle from each connected
  component of $G$ containing a cycle. Now define a Borel antimatching $f$ on
  these connected components as follows. 
  For each cycle $x_0, x_1, \ldots x_n = x_0$ in $C$, let $f(x_i) =
  x_{i+1}$ for $i < n$, and $f(x_n) = x_0$. Now use
  Lemma~\ref{extension_of_antimatchings} to extend $f$ to a total Borel
  antimatching $f^*$ on these connected components.
  
  So assume that $G$ is acyclic. By Lemma~\ref{antimatching_2reg}, we can
  find a partial Borel antimatching of $G$ such that $G \restriction
  \comp{(\dom(f))}$ is $2$-regular. Let $A = \comp{(\dom(f))}$. Now take a
  Borel set of edges of $G \restriction A$ that are pairwise disjoint and so
  that the set contains at least one
  edge from each connected component of $G \restriction A$. Remove these edges
  from $G$ to obtain the Borel graph $G'$ on $X$. Now using
  Lemma~\ref{antimatching_2reg} on $G'$, we may obtain another set $B
  \subset X$ that is the complement of a partial Borel antimatching on $G'$
  such that $G' \restriction B$ is $2$-regular. Note that $G\restriction A$
  and $G'\restriction B$ do not have any connected components that are equal.

  Now these $A$ and
  $B$ correspond to places where we have failed to construct antimatchings.
  Hence, without loss of generality, we may assume that each connected
  component of $G$ meets both $A$ and $B$. By
  Lemma~\ref{intersection_graph_coloring}, let $C$ be a Borel set of
  pairwise disjoint finite paths in $G$ from elements of
  $A$ to elements of $B$ that contains at least one path from every
  connected component of $G$. We may assume that if $x_0, \ldots, x_n$
  is a path in $C$, then $x_0$ is the only point of this path in $A$, and
  $x_n$ is the only point of this path in $B$. (We allow paths consisting
  of a single point where $A$ and $B$ intersect).
  Thus, each pair of connected
  components of $G \restriction A$ and $G' \restriction B$ are connected by at
  most one path in $C$, since $G$ is acyclic.
  
  Let $S \subset X$ consist of the connected components of $G \restriction A$
  that meet only finitely many paths in $C$. Since this set has a Borel
  transversal, we can obtain a Borel antimatching of $G \restriction S$. We
  can then use Lemma~\ref{extension_of_antimatchings} to extend this to a
  Borel antimatching of the connected components of $G$ that meet $S$. An
  identical comment is true for $B$. Thus, without loss of generality, we
  can assume that for each connected component of $G \restriction A$ and $G'
  \restriction B$, if there is a path in $C$ that meets this connected
  component, then there are infinitely many. 

  Let $Y$ be the
  collection of starting points of paths in $C$, and $Z$ be the collection
  of ending points of paths in $C$, so there is a canonical Borel bijection
  between $Y$ and $Z$. 
  Note that $Y$ and $Z$ may have nonempty
  intersection. Define $W = \{0\} \times Y \union
  \{1\} \times Z$. Consider the $3$-regular Borel graph $H$ on $W$, defined
  by the following three conditions. First, $(0,x) \mathrel{H} (1,y)$ if and only if
  there is a path in $C$
  from $x$ to $y$. Second, $(0,x) \mathrel{H} (0,y)$ if and only if there is a path from 
  $x$ to $y$ in $G \restriction A$ that does not contain any other element of
  $Y$. Third, $(1,x) \mathrel{H} (1,y)$ if and only if there is a
  path from $x$ to $y$ in $G' \restriction B$ that does not contain any other
  element of $Z$. $H$ is $3$-regular since connected components of $G
  \restriction A$ and $G' \restriction B$ that are met by paths in $C$ are
  met by infinitely many such paths.

  $H$ is a graph of the type we discussed at the beginning of this proof,
  and hence we can find a Borel antimatching of $H$. Let $A^* \subset A$ be
  the points that are in the same connected component of $G \restriction A$ as
  some element of $Y$. Let $B^* \subset B$ be the points that are in the
  same connected component of $G' \restriction B$ as some element of $Z$.
  It is clear that we can lift the Borel
  antimatching of $H$ to a partial Borel antimatching $f$ of $G$ whose
  domain is $A^* \union B^* \union \{x : \exists p \in C (x \in p)\}$, and
  such that $\ran(f) \subset \dom(f)$. We finish by applying
  Lemma~\ref{extension_of_antimatchings}.

  Our proof above has shown that (2) $\implies$ (3). We now show that 
  assuming that (2) is true modulo a nullset with respect to every Borel
  probability measure implies that (3) is true modulo a nullset with
  respect to every Borel probability measure. 
  
  Assume $G$ is a locally finite Borel graph on $X$ and $\mu$ is a Borel
  probability measure on $X$. Let $E_G$ be the connectedness relation for
  $G$. We can find a Borel probability measure $\nu$ which dominates
  $\mu$ and such that $\nu$ is $E_G$-quasi-invariant~{\cite[Section
  8]{MR2095154}}. Now perform the same
  process as above to obtain a pair of equivalence relations $E$ and
  $F$ on some Borel subset $Y$ of $X$, such that from Borel disjoint complete
  sections for $E$ and $F$, we can define a Borel antimatching of $G$. 
  
  Now this transformation of
  disjoint complete sections for $E$ and $F$ into an antimatching of $G$ is
  ``local'' in the sense that inside each connected component $C$ of $G$, we
  have a Borel way of transforming disjoint complete sections for $E
  \restriction Y \inters C$ and $F \restriction Y \inters C$ into an antimatching
  of $G \restriction C$.
  Hence, given disjoint Borel sets $A$ and $B$ such that $A$ meets
  $\nu$-a.e.\ $E$-class and $B$ meets $\nu$-.a.e.\ $F$-class, we can find a
  Borel antimatching of $G$ restricted to a Borel $\nu$-conull set, 
  since $\nu$ is $E_G$-quasi-invariant.
   
  Throughout the remainder of this proof, the same idea as above can be
  used to turn pure Borel implications between our four statements into
  implications in the measure context, and in the Baire category context.
  We leave it to the reader to perform the rest of these transformations.

  (3) $\implies$ (1)
  Let $E$ and $F$ be countable Borel equivalence relations such that every
  $E$-class has cardinality $\geq 3$ and every $F$ class has cardinality
  $\geq 2$. 
  By~{\cite[Proposition 7.4]{MR2095154}}, there exist Borel
  equivalence relations $E^* \subset E$ and $F^* \subset F$ such that every
  $E^*$-class is finite and has cardinality $\geq 3$ and every $F^*$-class
  is finite and has cardinality $\geq 2$. Hence, we may assume that 
  all the equivalence classes of $E$ and $F$
  are finite. Let $Y \disjointunion Z$ be the disjoint union of the
  equivalence classes of $E$ and the equivalence classes of $F$
  respectively. Let $G$ be the graph on $Y
  \disjointunion Z$ where $R$ and $S$ are adjacent in $G$ if $R \in Y$, $S
  \in Z$, and $R \inters S \neq \emptyset$.

  Now let $W$ be an uncountable standard Borel space, and
  extend $G$ to a locally finite Borel graph $G^*$ on $Y \disjointunion Z \disjointunion
  W$ so that every vertex in $Y \disjointunion W$ has degree
  $\geq 3$ in $G^*$, every vertex in $Z$ has degree $\geq 2$ in $G^*$, and such that $R \in Y
  \disjointunion Z$ is adjacent to an element of $W$ in $G^*$ if and only if $R \in
  Y$ and the degree of $R$ is $< 3$ in $G$ or $R \in Z$ and the degree of $R$ is
  $< 2$ in $G$. Note that for such $R$ there must be $S \in Y
  \disjointunion Z$ distinct from $R$ such that $R \inters S$ has cardinality $\geq 2$. 

  Now let $f$ be a Borel antimatching of $G^*$. Of course, $G^*$ does not
  have degree $\geq 3$. However, the neighbors of every degree $2$ vertex
  in $G^*$ all have degree $\geq 3$. Hence, we can contract away vertices of
  degree $2$, find a Borel antimatching of this graph, and then use it in
  the obvious way to find a Borel antimatching of $G^*$.
  
  Let $A_0$ be the set of $x \in X$ such that there exists $R \in Y$
  such that $f(R) \in Z$ and $R \inters f(R) = \{x\}$. Let $A_1$ be the
  Borel set of $x \in X$ such that there exists an $R \in Y$ and $S \in Z$
  such that $R \inters S$ has cardinality $\geq 2$, and $x$ is the least
  element of $R \inters S$. Let $A = A_0 \union A_1$. Clearly $A$ meets
  every equivalence class of $E$, and $\comp{A}$ meets every equivalence
  class of $F$.

  (3) $\implies$ (4) is obvious.

  (4) $\implies$ (2). 
  Suppose we have two independent aperiodic countable Borel equivalence relations $E$
  and $F$ on a standard Borel space $X$. By {\cite[Proposition
  7.4]{MR2095154}} we
  can find $E^*$ and $F^*$, finite Borel subequivalence relations of $E$
  and $F$ whose equivalence classes all have cardinality $n$. The
  intersection graph of their equivalence classes is Borel bipartite and
  $n$-regular. From a Borel antimatching for this graph, we can produce Borel disjoint
  complete sections for $E$ and $F$, as in the proof that (3)
  $\implies$ (1). 
\end{proof}

Our final goal will be to prove a
theorem about edge colorings for $3$-regular acyclic Borel bipartite graphs
in the context of measure and category. This will follow from 
several more equivalences extending those of Theorem~\ref{dcs_equiv_1} above. 

We begin with another definition.
Suppose $G$ is a graph on $X$. A \define{directing} of $G$ is a set
$D \subset G$ that contains exactly one of $(x,y)$ and $(y,x)$ for every
pair of neighbors $x,y \in X$.
A \define{partial directing} of $G$ is a
subset of $G$ that contains at most one of $(x,y)$ and $(y,x)$ for every
pair of neighbors $x,y \in X$. 
Given a partial directing $D$ of a graph $G$,
say that a point $x \in X$ is a \define{source} if $(x,y) \in D$ for some $y$, and
$(y, x) \notin D$ for all $y$. Similarly, say
that a point $x \in X$ is a \define{sink} if $(y,x) \in D$ for some $y$, and $(x,y)
\notin D$ for all $y$. 
Of course, if $f$ is an antimatching of a graph $G$ and 
we extend the set $\{(x,f(x)) : x \in X\}$ to a directing $D$ of $G$,
then this directing will have no sinks. 

\begin{lemma}\label{extension_of_directings}
  Suppose that $G$ is a locally countable Borel graph such that each vertex
  of $G$ has degree $\geq 2$, and $D$ is a partial Borel directing of $G$
  without sources or sinks. 
  Suppose also that every connected component of $G$ contains at least one
  vertex that is incident to an edge of $D$. Then $D$ can be extended to a
  total Borel directing $D^*$ of $G$ that has no sources or sinks.
\end{lemma}
\begin{proof}
  Suppose that $x_0, x_1,
  \ldots, x_n$ is a path in $G$ such that $x_0$ and $x_n$ are both
  incident to edges already in $D$. Then we can extend $D$ by adding the
  edges from $x_0, x_1, \dots, x_n$ that do not conflict with edges
  already in $D$; add $(x_i, x_{i+1})$ to $D$ unless $(x_{i+1}, x_i)$ is
  already in $D$. The property that $D$ has no sources
  or sinks is preserved when we add paths in this way. Similarly, given a
  cycle, we can extend $D$ using this cycle in the analogous way,
  while preserving the property that $D$ has no sources or sinks.

  Use Lemma~\ref{intersection_graph_coloring} to partition all
  the finite paths and cycles of $G$ 
  into countably many Borel sets
  $\{P_i\}_{i \in \N}$ such that the elements of each $P_i$ are pairwise
  disjoint. Let $k_0, k_1, \ldots$ be a sequence that contains
  each element of $\N$ infinitely many times.
  Let $D_0 = D$.
  Now define $D_{i+1}$ from $D_i$ by extending $D_i$ via all the cycles of
  $P_{k_i}$, and all the paths of
  $P_{k_i}$ that start and end at vertices incident to at least one edge in
  $D_i$. Let $D_\infty = \bigunion_{i \in \N} D_i$. 

  Let $A$ be the set of vertices that are incident to at least one edge in
  $D_\infty$. It is clear that if $x_0 \in A$ and $x_0, x_1, \ldots, x_n$
  is a path in $G$, then $x_n \in A$ implies that $x_i \in A$ for all $0 \leq i
  \leq n$. 
  
  We finish by extending $D_\infty$ to $D^*$ by directing the remaining
  edges of $G$ ``away'' from $D_\infty$. More precisely, let $x$ and $y$
  be distinct elements of $X$ and suppose 
  that neither $(x,y)$ nor $(y,x)$ are in $D_\infty$. Then there must be a
  unique path $x_0, \ldots, x_n$ such that $x_0$ is incident to
  an edge in $D_\infty$, and the path ends with $(x_{n-1}, x_n)$ equal to $(y,x)$ or
  $(x,y)$. Extend $D_\infty$ to $D^*$ by adding all such $(x_{n-1},x_n)$.
\end{proof}

\begin{thm}\label{dcs_equiv_2}
  The fallowing statements are all false in the full Borel context. They are
  true modulo a nullset with respect to any Borel probability measure, and
  true modulo a meager set with respect to any compatible Polish topology. 
  \begin{enumerate}
    \item Every pair of countable Borel equivalence relations $E$ and
    $F$ on a standard Borel space $X$ 
    such that the $E$-classes all have cardinality $\geq 3$ and the $F$-classes
    all have cardinality $\geq 2$ admits disjoint Borel complete
    sections. 
    \item For every pair of aperiodic countable Borel
    equivalence relations $E$ and $F$ on a standard Borel space $X$, there
    exists a Borel set $B \subset X$ such that 
    $B$ and $\comp{B}$ are complete sections for both $E$ and
    $F$.
    \item Every $3$-regular Borel graph has a directing with no sinks or
    sources. 
    \item Every $3$-regular Borel bipartite graph has a 
    Borel edge coloring with $4$ colors.
  \end{enumerate}
\end{thm}

\begin{proof}
  (1) is false by Theorem~\ref{nodcs} and true in the measure and category
  context by Theorem~\ref{dcs_equiv_1}. We will use the same type of
  proof as Theorem~\ref{dcs_equiv_1}. 

  (1) $\implies$ (2). Since $E$ and $F$ are aperiodic, the argument that
  (3) $\implies$ (1) in Theorem~\ref{dcs_equiv_1} produces subequivalence
  relations $E^*$ and $F^*$ of $E$ and $F$ with finite classes, and a Borel
  set $A$ such that $A$ and $\comp{A}$ are complete sections for $E^*$, and
  $\comp{A}$ is a complete section for $F^*$. Hence, $A$ meets every
  $E$-class, and $\comp{A}$ meets every $E$-class and every $F$-class in
  infinitely many places. Thus, if we run the same argument on the
  aperiodic equivalence relations $E \restriction \comp{A}$ and $F
  \restriction
  \comp{A}$ with their roles reversed, we obtain a Borel set $A'
  \subset \comp{A}$ such that $A'$ meets every $F \restriction
  \comp{A}$-class, and $\comp{(A')}$ meets both every $E \restriction
  \comp{A}$-class and every $F \restriction \comp{A}$-class. Now let $B = A
  \union A'$. 

  (2) $\implies$ (1) follows from Theorem~\ref{dcs_equiv_1}.

  (2) $\implies$ (3). Let $G$ be a $3$-regular Borel graph. Using Lemma~\ref{extension_of_directings}, we may assume that $G$ is
  acyclic, as in the proof of (2) $\implies$ (3) for Theorem~\ref{dcs_equiv_1}. 

  We begin by letting $Y \subset [X]^2$ be
  a Borel set of pairwise disjoint edges of
  $G$ that contains at least one edge from each connected component of $G$.
  We define two countable Borel equivalence
  relations $E$ and $F$ on $Y$ as follows: $R$ and $S$ are related by $E$
  if their least points are connected in $G \setminus Y$, and related by
  $F$ if their greatest points are connected in $G \setminus Y$. Here we
  use $G \setminus Y$ to denote the
  graph $G$ with the edges from $Y$ removed.

  We may assume that all the equivalence classes of $E$ and $F$ are
  infinite; on the connected components of 
  $G \setminus Y$ that correspond to 
  equivalence classes of $E$ and $F$ that are finite, and we can apply
  Lemma~\ref{extension_of_directings} to get a directing of the connected
  components of $G$ containing points corresponding to 
  finite $E$-classes or $F$-classes.

  Now let $B \subset Y$ be a Borel set such that both $B$ and $\comp{B}$ are complete
  sections for $E$ and $F$. Let $D_0 = \{(x,y) : \{x,y\} \in B \text{ and $x$
  is less than $y$}\}$. Each connected component of $G \setminus Y$
  contains infinitely many $x$ such that $(x,y) \in D_0$ for some
  $y$, and infinitely many $y$ such that $(x,y) \in D_0$ for some $x$.
  We will extend $D_0$ to a total Borel directing of $G$
  without sinks or sources.
  
  Consider the set of paths $x_0, x_1, \ldots, x_n$ in $G \setminus Y$ such
  that there exists $y$ and $z$ such that both $(y,x_0)$ and $(x_n,z)$ are
  in $D_0$. 
  We may use Lemma~\ref{intersection_graph_coloring} to
  partition these paths into countably many Borel sets $\{P_i\}_{i \in \N}$
  such that the elements of each $P_i$ are pairwise disjoint. Now as in the proof of
  Lemma~\ref{extension_of_directings}, for each $i
  \in \N$, extend each $D_{i}$ to $D_{i+1}$ by adding the edges from
  the paths of $P_i$ which do not conflict with edges already in $D_i$. Let
  $D_\infty = \bigunion_{i \in \N} D_i$. Then complete $D_\infty$ to a
  total directing $D$ using Lemma~\ref{extension_of_directings}.

  (3) $\implies$ (1) follows from Theorem~\ref{dcs_equiv_1}. It is clear
  that such a directing can be used to define a Borel antimatching.

  (3) $\implies$ (4). Suppose that $G$ is a Borel bipartite $3$-regular
  graph whose bipartiteness is witnessed by the Borel sets $A$ and $B$.
  Suppose $D$ is a Borel directing of $G$ without sinks or sources. We can
  use $D$ to write $G$ as the disjoint union of two graphs $H_0$ and $H_1$
  in the following way: the edges of $H_0$ are those directed by $D$ from
  $A$ to $B$, and the edges of $H_1$ are those directed by $D$ from $B$ to
  $A$. The vertices in $H_0$ and $H_1$ all have degree $1$ or $2$. Hence,
  each connected component of the $H_i$ is finite, a ray (having
  exactly one vertex of degree $1$), or a line (having no vertices of
  degree $1$).

  If all the connected components of $H_0$ and $H_1$ were finite or rays,
  then it
  would be trivial to construct a Borel edge coloring of $G$ with four colors;
  we could simply edge color $H_0$ using the colors $\{0,1\}$, edge color
  $H_1$ using the colors $\{2,3\}$, and then take the union of these
  colorings. Our problem is that in general, we will need to use $3$
  colors in an edge coloring of an $H_i$ containing lines.

  Let $Y \subset [X]^2$ be a Borel set of pairwise disjoint 
  edges from $H_0$ consisting of infinitely many edges from each line
  in $H_0$. Define the countable
  Borel equivalence relations $F_0$ and $F_1$ on $Y$ where $S$ and $R$ are
  $F_i$-related if there exist $x \in S$ and $y \in R$ that are in the same
  connected component of $H_i$. 
  Clearly every equivalence class of $F_0$ is infinite, however, there may
  be equivalence classes of $F_1$ that are finite. 
  
  Now take a Borel set $C
  \subset Y$ that is a complete section for $F_0$, so that $\comp{C}$ meets
  every infinite equivalence class of $F_1$. We can find such a $C$ by
  letting $Z$ be an uncountable standard Borel space and 
  extending $F_0$ and $F_1$ to aperiodic equivalence relations $F_0^*$ and $F_1^*$ on
  $Y \disjointunion Z$ such that if $x \in Z$ and $y \in Y$, then
  $x \cancel{F_0^*} y$ and 
  $x F_1^* y$ only if $[y]_{F_1}$ is finite. Now find disjoint complete
  sections for $F_0^*$ and $F_1^*$.
  
  Let $H_0^*$ be the graph $H_0$ but with the edges from $C$ removed, and
  let $H_1^*$ be the graph $H_1$ but with the edges from $C$ added. Clearly
  $H_0^*$ has no lines. Further, all the lines that we have added to
  $H_1^*$ must contain rays from $H_1$. This is because the elements of $C$
  in a new line in $H_1^*$ must all be $F_1$-related, and therefore come
  from an $F_1$-class that is finite. Hence, we can edge-color these lines
  from $H_1^*$ in a Borel way with $2$-colors. 

  If we perform the same process again with $H_1^*$ and $H_0^*$ in lieu of
  $H_0$ and $H_1$, respectively, 
  then we obtain Borel graphs $H_0^{**}$ and $H_1^{**}$ such that $G =
  H_0^{**} \union H_1^{**}$ and both $H_0^{**}$ and $H_1^{**}$ have Borel
  edge colorings with $2$ colors.

  (4) $\implies$ (1). We use Theorem~\ref{dcs_equiv_1} again. Let $G$ be a
  Borel bipartite $3$-regular graph, whose bipartiteness is witnessed by
  the Borel sets $A$ and $B$. Suppose that $G$ has a Borel edge coloring
  with $4$ colors. We can use this coloring to define a Borel antimatching
  of $G$. First, partition the four colors into the sets $\{0,1\}$ and
  $\{2,3\}$. Notice that each vertex must be incident to at least one edge
  of color $0$ or $1$, and at least one edge of color $2$ or $3$. Thus, we
  can define a Borel antimatching by setting $f(x) = y$ if $x \in A$ and
  $y$ is the least neighbor of $x$ such that $(x,y)$ is colored $0$ or
  $1$, or if $x \in B$ and $y$ is the least neighbor of $x$ such that
  $(x,y)$ is colored $2$ or $3$. 
\end{proof}

As a consequence of the above lemma, we obtain the following: 

\begin{thm}\label{measurable_3regular_edge}
  Suppose $G$ is a Borel bipartite $3$-regular graph on $X$. Then $G$ has a
  Borel edge coloring with $4$ colors modulo a null set or meager set with
  respect to any Borel probability measure on $X$ or Polish topology
  realizing the standard Borel structure of $X$. 
\end{thm}

The full measurable analogue of Vizing's theorem for Borel graphs remains
open.

\begin{question}
Given any $n$-regular Borel graph $G$ on a standard Borel
probability space $(X,\mu)$, must there be a $\mu$-measurable edge coloring
of $G$ with $n+1$ colors?
\end{question}

\end{document}